\documentclass[12pt,centertags,oneside]{amsart}
\usepackage{amsmath,amstext,amsthm,amscd,typearea}
\usepackage{amssymb}
\usepackage{a4wide}
\usepackage[mathscr]{eucal}
\usepackage{mathrsfs}
\usepackage{typearea}
\usepackage{charter}
\usepackage{pdfsync}
\usepackage{url}

\usepackage[a4paper,width=16.2cm,top=3cm,bottom=3cm]{geometry}

\numberwithin{equation}{section}

\newtheorem{theorem}{Theorem}[section]
\newtheorem{definition}[theorem]{Definition}
\newtheorem{proposition}[theorem]{Proposition}
\newtheorem{corollary}[theorem]{Corollary}
\newtheorem{lemma}[theorem]{Lemma}

\newcommand{\cali}[1]{\mathscr{#1}}

\newcommand{\dist}{\mathop{\mathrm{dist}}\nolimits}

\def\mF{\mathcal{F}}

\newcommand{\Cc}{\cali{C}}

\newcommand{\Pc}{\cali{P}}

\newcommand{\Uc}{\cali{U}}

\newcommand{\capK}{\rm cap}

\newcommand{\B}{\mathbb{B}}
\newcommand{\C}{\mathbb{C}}

\newcommand{\N}{\mathbb{N}}

\newcommand{\R}{\mathbb{R}}

\title[]{Families of  Monge-Amp\`ere measures with H\"older continuous potentials}

\author{Duc-Viet Vu}
\address{Korea institute for advanced study,
 85 Hoegiro, Dongdaemun-gu, Seoul 02455, Republic of Korea}
\email{vuviet@kias.re.kr}

\date{\today}
\begin{document}

\begin{abstract} Let $X$ be a compact K\"ahler manifold of dimension $n.$ Let $\mF$ be a family of  probability measures on $X$ whose superpotentials are of  uniformly bounded  $\Cc^\alpha$ norms for some fixed constant $\alpha \in (0,1].$  We prove that the corresponding family of solutions of the complex Monge-Amp\`ere equations $(dd^c \varphi + \omega)^n= \mu$ with $\mu \in \mF$ is H\"older continuous. 
\end{abstract}

\maketitle

\medskip

\noindent

\medskip


\tableofcontents

\section{Introduction} \label{introduction}

Let $X$ be a compact K\"ahler manifold of dimension $n$ with a fixed K\"ahler form $\omega$ so normalized that $\int_X \omega^n=1.$ Let $\mu$ be a probability measure on $X.$  For  every bounded $\omega$-psh function $\varphi$ on $X$ and $1 \le j \le n,$ we put $\omega_\varphi^j:= (dd^c \varphi+ \omega)^j$ which is well-defined by \cite{Bedford_Taylor_76,Klimek}. Consider the  complex Monge-Amp\`ere equation 
\begin{align}\label{eq_MA}
\omega^n_\varphi= \mu,
\end{align}
where $\varphi$ is a bounded $\omega$-psh function on $X$ and $\int_X  \varphi \, \omega^n= 0.$ 
The equation (\ref{eq_MA}) and its variants have been extensively studied and have a  wide range of applications. Instead of giving details on the development of the research on (\ref{eq_MA}), in this short paper,  we refer  the readers to \cite{Yau1978,kolodziej05,Kolodziej_Acta,Kolodziej08holder,DemaillyHiep_etal,Dinew_Zhang_stability,Hiep_holder,Dinew_09,Phong,Blocki_stability} and the references therein for detailed information. 

In this work, we study the H\"older continuity of solutions of (\ref{eq_MA}). Recently, based on  \cite{DemaillyHiep_etal},  Dinh and Nguy\^en proved  in \cite{DinhVietanhMongeampere} that  (\ref{eq_MA}) has a unique H\"older continuous solution $\varphi_{\mu}$  if and only if $\mu$ has a H\"older continuous superpotential $\Uc_{\mu}$, see Definition \ref{def_superholder} below. Precisely, they proved that if $\Uc_{\mu}$ is H\"older continuous with H\"older exponent $\alpha \in (0,1],$ then $\varphi_{\mu} \in \Cc^{\beta}(X)$ for any $\beta \in (0, \frac{2\alpha}{n+1}),$ where $\Cc^{\beta}(X)$ denotes the set of H\"older continuous functions with H\"older exponent $\beta$ on $X.$ In this case, we call $\varphi_{\mu}$ \emph{the (Monge-Amp\`ere) potential of $\mu$}.  
In view of the last result, we would like to address the question of  the stability of the H\"older continuity of  the solution of  (\ref{eq_MA})  with respect to $\mu:$ given a family of probability measures with H\"older continuous superpotentials,  does $\varphi_\mu$ depend H\"older continuously on $\mu$ in that family? Let us be more clear in the next paragraph.

Let $\alpha \in (0,1].$ By  \cite[Pro. 4.1]{DinhVietanhMongeampere}, if $\varphi$ varies in a  bounded subset of  $\Cc^\alpha(X),$ then  $\omega_\varphi^n$ has a H\"older continuous superpotential  with uniformly bounded H\"older exponent and H\"older constant. Hence, in order to study the above stability problem, it is necessary to consider sets of probability measures having the last property.   

Now  let $\Pc_{\alpha}$ be a set of probability measures $\mu$ on $X$ such that  the superpotential  $\cali{U}_{\mu}$ of every $\mu \in \Pc_{\alpha}$ is H\"older continuous with  H\"older exponent $\alpha$ and a  H\"older constant independent of $\mu.$  Let  $\beta \in (0, \frac{2\alpha}{n+1}).$ Define $\Phi: \Pc_{\alpha} \rightarrow  \Cc^{\beta}(X)$ by sending $\mu \in \Pc_{\alpha}$ to the unique solution $\varphi_{\mu}$ of (\ref{eq_MA}).   Recall that the set of probability measures on $X$ endowed with the weak topology is a metric space with the distance $\dist$ defined as follows: for measures  $\mu,\mu',$
$$\dist(\mu,\mu'):= \sup_{\|v\|_{\Cc^{1}} \le 1} \big| \langle \mu-\mu', v \rangle \big|,$$
where $v$ is a smooth real-valued function on $X.$  The following is our main result. 

\begin{theorem} \label{th_stability}  The map  $\Phi$ is H\"older continuous with H\"older exponent   $\alpha'$ for any $0 < \alpha' < \beta(\frac{2\alpha}{n+1}- \beta)  2^{-n-1}$.   
\end{theorem}

Equivalently, the last theorem says that there is a constant $C$ (depending on $\beta, \alpha'$) such that 
\begin{align} \label{ine_holderphimu}
\| \varphi_{\mu_1} -  \varphi_{\mu_2}\|_{\Cc^{\beta}} \le C [\dist(\mu_1, \mu_2)]^{\alpha'},
\end{align}
for every $\mu_1,\mu_2 \in \Pc_{\alpha}.$  Consequently,  if $\{\mu_k\}_{k\in \N} \in \Pc_{\alpha}$ converges weakly to $\mu \in  \Pc_{\alpha},$ then the associated solution $\varphi_{\mu_k}$ converges to $\varphi_{\mu}$ in $\Cc^{\beta}(X).$  An interesting feature in the last assertion is that  $\mu_k$ and $\mu$ can be singular to each other for every $k$. An imitation of Ko{\l}odziej's arguments in \cite{Kolodziej_2003} only gives an estimate of type (\ref{ine_holderphimu}) but with $\dist(\mu_1,\mu_2)$ replaced by the mass norm $\|\mu_1- \mu_2\|$ of $(\mu_1- \mu_2)$.  A such estimate is not useful when $\mu_1,\mu_2$ are singular to each other. For example as in the situation described in Corollary \ref{cor_MApotenhol} below, the supports of measures $\mu_1,\mu_2$  in question are disjoint, hence  $\|\mu_1 - \mu_2\|=2$ in this case. 

We give now an application of our main result. Recall that a  real submanifold of $X$ is said to be \emph{Cauchy-Riemann generic}   if the real tangent space at any point of it isn't contained in a complex hypersurface of the real tangent space at that point of $X.$ By \cite{Vu_MA}, the restriction of a smooth volume form  of an immersed (Cauchy-Riemann) generic  submanifold $Y$ of $X$ to a compact subset $K$ of $Y$ has a H\"older continuous superpotential. It is also clear from the arguments there that if the compact $K$ depends smoothly on a parameter $\tau$ then the H\"older exponent and H\"older constant of the superpotential can be chosen to be fixed numbers for every $\tau,$ see Proposition \ref{pro_uniformrealsub} below. Precisely,  let $M$ be a compact real manifold and  $Y$ a real Riemannian manifold. Assume that there is  a  smooth map  $\Psi: Y \times M \rightarrow X$ such that  $\Psi_\tau:= \Psi|_{Y \times \{\tau\}}: Y \rightarrow X$ is an embedding into $X$ such that $Y_\tau:= \Psi_\tau(Y)$ is a generic submanifold $Y_{\tau}$ for every $\tau \in M.$ Then   $\{Y_{\tau}\}_{\tau \in M}$ is a smooth family of generic submanifolds of $X.$ Note that using local charts of $X,$ we see that  such family exists abundantly. With this setting, we get the following nice geometric result.

\begin{corollary} \label{cor_MApotenhol} Let $K$ be a compact subset of $Y.$ For $\tau \in M,$ define $\mu_\tau$ to be the pushforward measure of the volume form of $Y$ on $K$ under $\Psi_\tau.$ Then the family of the Monge-Amp\`ere potential $\varphi_{\mu_\tau}$ of $\mu_\tau$ is H\"older continuous in $\tau.$
\end{corollary}

Note that as in Theorem \ref{th_stability}, we can give an explicit H\"older exponent in Corollary  \ref{cor_MApotenhol}.  In the next section, we will give a proof of Theorem \ref{th_stability}. 

\medskip
\noindent
{\bf Acknowledgement.} The author would like to thank Ngoc Cuong Nguyen for fruitful discussions.


\section{Proof of Theorem \ref{th_stability}} \label{sec_Phiholder}

Let $\mathcal{P}_0$ be the set of $\omega$-psh functions $\varphi$ on $X$ such that $\int_X  \varphi \, \omega^n= 0.$ We define the distance  $\dist_{L^1}$ on $\mathcal{P}_0$ by putting 
$$\dist_{L^1}(\varphi_1,\varphi_2):= \int_X |\varphi_1 - \varphi_2| \, \omega^n,$$ 
for every $\varphi_1, \varphi_2 \in \mathcal{P}_0.$ 

\begin{definition} \label{def_superholder} The superpotential of a probability measure $\mu$ (of mean $0$) is the function $\cali{U}:  \mathcal{P}_0 \rightarrow \R$ given by $\cali{U}(\varphi):= \int_X \varphi d\mu.$ We say that $\cali{U}$ is H\"older continuous with H\"older exponent $\alpha \in (0,1]$ if it is so with respect to the distance $\dist_{L^1}.$ The $\Cc^\alpha$-norm of $\Uc$ is defined as usual.     
\end{definition}

By \cite[Le. 3.3]{DinhVietanhMongeampere}, that $\Uc$ is H\"older continuous with H\"older exponent $\alpha \in (0,1]$ is equivalent to having
\begin{align} \label{le_ine_trunhau}
\int_X|\varphi_1 - \varphi_2| d\mu \le C \max \big \{ \|\varphi_1 -\varphi_2\|_{L^1(X)}^{\alpha},\|\varphi_1 -\varphi_2\|_{L^1(X)} \big \},
\end{align} 
for some constant $C$ independent of $\varphi_1, \varphi_2.$  
By the arguments in \cite{DinhVietanhMongeampere}, we immediately get the following.

\begin{lemma} \label{le_uniformbound} Assume that  the superpotential $\cali{U}$ of  a probability measure $\mu$ on $X$ is H\"older continuous with H\"older exponent $\alpha$ and H\"older constant $C.$ Let  $\beta \in (0, \frac{2\alpha}{n+1}).$  Then the unique solution $\varphi_\mu$ of (\ref{eq_MA}) with $\int_X  \varphi_\mu \, \omega^n= 0$ is H\"older continuous with H\"older exponent $\beta$ and H\"older constant $\tilde{C}$ depending only on $\alpha,\beta,C$ and $X.$ In particular, $\varphi_\mu$ is bounded by $\tilde{C}$ independent of $\mu \in \Pc_\alpha.$
\end{lemma}

Let $K$ be a Borel subset of $X.$ The \emph{capacity} of $K$ is given by 
$$\capK_\omega(K):= \sup \big\{ \int_K \omega_\varphi^n: 0 \le \varphi \le 1, \varphi \, \text{ $\omega$-psh} \big \}.$$
The above notion is due to Ko{\l}odziej as an analogue to the capacity given by Bedford and Taylor in the  local setting. As introduced in \cite{DinhVietanhMongeampere}, a positive measure $\mu$ is said to be \emph{K-moderate} if  there are positive constants $A$ and $\delta_0$ for which 
\begin{align}\label{ine_capdominate}
\int_K \omega_{\varphi_1}^n \le A  e^{- [\capK_\omega(K)]^{-\delta_0}}.  
\end{align}
for every Borel subset $K$ of $X.$  Recall that if $\mu$ has a H\"older continuous superpotential with H\"older exponent $\alpha$ and H\"older constant $C,$ then $\mu$ is K-moderate by \cite[Pro. 2.4]{DinhVietanhMongeampere}. Moreover,  the constants $A,\delta_0$ in (\ref{ine_capdominate})  depend only on $\alpha, C$ and $X.$    The following result is crucial for our later proof.  

\begin{lemma} \label{le_cap} Let $\varphi_1, \varphi_2$ be bounded $\omega$-psh functions on $X.$ Let $s$ be a real number. Assume that the set $\{\varphi_1 - s < \varphi_2\}$ is nonempty and $\omega_{\varphi_1}^n$ is K-moderate. Let $\delta$ be a positive number in $(0,1).$ Then there exists a constant $A'$ depending only on $A, \delta_0,\delta$ and $n$ such that for any $\epsilon \in (0,1)$ we have 
$$\capK_\omega \big(\{\varphi_1 - s- \epsilon < \varphi_2\}\big) \ge  A' (1+ \|\varphi_2\|_{L^{\infty}})^{-\delta} \epsilon^\delta.$$
\end{lemma}  

\proof  By (\ref{ine_capdominate}), we have
\begin{align}\label{ine_capdominate2}
\int_K \omega_{\varphi_1}^n \le A_1  [\capK_\omega(K)]^{-n^2 \delta^{-1}},  
\end{align}
for some constant $A_1$ depending only on $A,n,\delta_0,\delta.$ Define $h(t):= t^{n^2 \delta^{-1}}$ for positive real numbers $t.$  Put $c_\epsilon:= \capK_\omega \big(\{\varphi_1 - s- \epsilon < \varphi_2\}\big).$  Applying now \cite[Le. 2.2]{Kolodziej_2003} to $h(t)$ and $\varphi_1, \varphi_2$ gives 
$$\int_{c_\epsilon^{-1/n}}^{\infty} t^{-1} h^{-1/n}(t) dt + h^{-1/n}(c_\epsilon^{-1/n}) \gtrsim (1+ \|\varphi_2\|_{L^{\infty}})^{-1} \epsilon.$$
Then the desired inequality follows easily.  The proof is finished.
\endproof

\begin{lemma} \label{le_uocluongL1} Let $\mu_1, \mu_2 \in \Pc_\alpha$ and $\varphi_1, \varphi_2$ H\"older continuous solutions of  (\ref{eq_MA}) for $\mu_1,\mu_2$ respectively.  Let  $\beta \in (0, \frac{2\alpha}{n+1}).$ Then we have 
\begin{align} \label{ine_chuanL1}
\| \varphi_1 - \varphi_2\|_{L^1(X)} \le  C \dist(\mu_1, \mu_2)^{\beta 2^{-n}},
\end{align}
for some constant $C$ independent of $\mu_1,\mu_2.$
\end{lemma}

\begin{proof} 
By \cite[Th. 1.2]{Blocki_stability}, we have
\begin{align} \label{ine_ddbar}
\int_X d(\varphi_1 - \varphi_2) \wedge d^c (\varphi_1 - \varphi_2) \wedge \omega^{n-1} \le C\big( \int_X  (\varphi_1 - \varphi_2)(\omega^n_{\varphi_2} - \omega^n_{\varphi_1}) \big)^{2^{1-n}},
\end{align}
for some constant $C$ independent of $\mu_1,\mu_2.$  
 Now using Poincar\'e's inequality (see \cite[Th. 1, page 275]{Evans}) for $L^2$-norm and the fact that $\varphi_1, \varphi_2$ are H\"older continuous with H\"older exponent $\beta$ and a fixed H\"older constant, we get 
\begin{align} \label{ine_L1alpha}
 \| \varphi_1 - \varphi_2\|_{L^1(X)} \lesssim \big( \int_X  (\varphi_1 - \varphi_2)(\omega^n_{\varphi_2} - \omega^n_{\varphi_1}) \big)^{2^{-n}} \lesssim  \dist_\beta(\mu_1,\mu_2)^{2^{-n}},
 \end{align}
where 
$$\dist_\beta(\mu_1,\mu_2):=  \sup_{\|v\|_{\Cc^{\beta}} \le 1} \big| \langle \mu_1-\mu_2, v \rangle \big|.$$
Recall from \cite{DinhSibony_Pk_superpotential,Triebel} that $\dist_\beta(\mu_1, \mu_2)  \lesssim \dist^{\beta}(\mu_1, \mu_2)$ for $\beta \in [0,1].$ This together with (\ref{ine_L1alpha}) gives (\ref{ine_chuanL1}). The proof is finished.
\end{proof}

The following result is the interpolation inequality for H\"older norms of which we include a proof for the readers' convenience. 

\begin{lemma} \label{le_danhgiachuanholder} Let $f$ be a H\"older continuous function in $\Cc^\beta(X)$ for some positive constant $\beta \in (0,1).$ Let $\epsilon$ be a positive number in $[0,1-\beta].$ Then we have
$$\| f\|_{\Cc^{\beta}} \le C  \|f\|_{\Cc^0}^{\frac{\epsilon}{\beta+ \epsilon}} \|f\|_{\Cc^{\beta+\epsilon}}^{\frac{\beta}{\beta+ \epsilon}},$$
for some constant $C$ depending only on $X.$
\end{lemma}

\begin{proof} Recall that the $\Cc^\beta$-norm is defined in $X$ by using a fixed cover of $X$ by  local charts. Thus without loss of generality, we can 	assume that $X= \C^n.$ For $x, y \in \C^n,$ we have 
\begin{align*}
\frac{|f(x)- f(y)|}{|x-y|^{\beta}}= |f(x)- f(y)|^{\frac{\epsilon}{\beta+ \epsilon}} \big(  \frac{|f(x)- f(y)|}{|x-y|^{\beta+\epsilon}}  \big)^{\frac{\beta}{\beta+ \epsilon}}\le  2 \|f\|_{\Cc^0}^{\frac{\epsilon}{\beta+ \epsilon}} \|f\|_{\Cc^{\beta+\epsilon}}^{\frac{\beta}{\beta+ \epsilon}}.
\end{align*}
The proof is finished.
\end{proof}

\begin{proof}[End of the proof of Theorem  \ref{th_stability}] Let $\mu_1, \mu_2 \in \Pc_\alpha$ ($\mu_1 \not = \mu_2$) and $\varphi_1, \varphi_2$ H\"older continuous solutions of  (\ref{eq_MA}) for $\mu_1,\mu_2$ respectively.  Fix a constant  $\beta \in (0, \frac{2\alpha}{n+1})$ and $\delta \in [0, \frac{2 \alpha}{n+1}-\beta).$  By Lemma \ref{le_uniformbound} and the definition of $\Pc_\alpha,$ there is a positive constant $\tilde{C}$ independent of $\varphi_1, \varphi_2$ such that $\varphi_1,\varphi_2$ are H\"older continuous with H\"older exponent $(\beta+\delta)$ and H\"older constant $\tilde{C}.$   Set 
$$N(\varphi_1, \varphi_2):= \max\{ \| \varphi_1 - \varphi_2\|^{\alpha}_{L^1(X)},  \| \varphi_1 - \varphi_2\|_{L^1(X)} \} \not = 0.$$
Fix   a real number $\tilde{\delta}$ in $(0,1).$ In order to prove (\ref{ine_holderphimu}), it suffices to suppose from now on that $\dist(\mu_1,\mu_2)$ is small. As it will be clear later, we will need that  $\dist(\mu_1,\mu_2)$ is less than a  positive constant depending on $\tilde{\delta}$ but independent of $\mu_1, \mu_2.$  By  Lemma \ref{le_uocluongL1},  the quantity $N(\varphi_1, \varphi_2)$ is also small.   In what follows, we use the notations $\lesssim$ and $\gtrsim$ to indicate $\le$ and $\ge$ respectively up to a multiplicative constant independent of $\mu_1, \mu_2.$ 

Let $\epsilon$ be a positive real number in $(0,1)$ to be chosen later.  Put $E_\epsilon:= \{\varphi_1 + \epsilon < \varphi_2 \}.$ On $E_\epsilon$ we have $\varphi_1 - \varphi_2 \le -\epsilon <0,$ hence $|\varphi_1 - \varphi_2| \ge  \epsilon.$ It follows that 
\begin{align} \label{ine_Eepsilonmu1}
\int_{E_\epsilon} d \mu_1 \le  \epsilon^{-1} \int_X |\varphi_1 - \varphi_2| d\mu_1 \lesssim \epsilon^{-1} N(\varphi_1, \varphi_2) 
\end{align}
by (\ref{le_ine_trunhau}). 
Since $|\varphi_2| \le \tilde{C},$ for any $\omega$-psh function $\varphi$ on $X$ such that $0 \le \varphi \le 1,$ we have 
$$E:= \{\varphi_1 + (\tilde{C}+2)\epsilon < \epsilon \varphi + (1- \epsilon) \varphi_2 \} \subset \{\varphi_1 +  (\tilde{C}+2)\epsilon < \epsilon+ \varphi_2+ \epsilon \tilde{C}\}= E_{\epsilon}.$$
This combined with the comparison principle gives
$$\int_{E} \omega_{\epsilon \varphi + (1- \epsilon) \varphi_2}^n \le \int_E \omega_{\varphi_1}^n \le \int_{E_{\epsilon}} \omega_{\varphi_1}^n= \int_{E_{\epsilon}} d\mu_1.$$
On the other hand, we also have $E_{2\epsilon(\tilde{C}+1)} \subset E$ and $\omega_{\epsilon \varphi + (1- \epsilon) \varphi_2}^n \ge \epsilon^n \omega^n_\varphi.$ This yields
$$ \epsilon^n \int_{E_{2\epsilon(\tilde{C}+1)}} \omega_\varphi^n \le \int_{E} \omega_{\epsilon \varphi + (1- \epsilon) \varphi_2}^n \le  \int_{E_{\epsilon }} d\mu_1.$$
Combining the last inequality with (\ref{ine_Eepsilonmu1}), we obtain
$$\epsilon^n \int_{E_{2\epsilon(\tilde{C}+1)}} \omega_\varphi^n \lesssim \epsilon^{-1}N(\varphi_1, \varphi_2).$$
Taking the supremum over every $\varphi$ in the last inequality implies
\begin{align} \label{ine_boundcap}
 \capK_\omega\big(E_{2\epsilon(\tilde{C}+1)}\big) \lesssim \epsilon^{-n-1} N(\varphi_1, \varphi_2).
\end{align}  
 Choose $\epsilon:= C_{\tilde{\delta}} \big(N(\varphi_1, \varphi_2)\big)^{1/(n+1+ \tilde{\delta})} \in (0,1),$ where $C_{\tilde{\delta}}>1$ is a constant big enough (depending on $\tilde{\delta}$) which is independent of $\varphi_1, \varphi_2$.
 Here recall that $N(\varphi_1, \varphi_2)$ was assumed to be small enough at the beginning of the proof.  

We claim that $E_{2\epsilon(\tilde{C}+2)}$ is empty.  Suppose the contrary. Thus applying Lemma \ref{le_cap} to $s:= -2(\tilde{C}+2) \epsilon$ shows that  $ \capK_\omega\big(E_{2\epsilon(\tilde{C}+1)}\big) \ge A_{\tilde{\delta}} \epsilon^{\tilde{\delta}}$ for some constant $A_{\tilde{\delta}}$ independent of $\varphi_1, \varphi_2.$   This coupled with (\ref{ine_boundcap})  gives 
\begin{align} \label{ine_boundNphi}
N(\varphi_1, \varphi_2) \gtrsim A_{\tilde{\delta}}  \epsilon^{n+1+ \tilde{\delta}}= A_{\tilde{\delta}} C_{\tilde{\delta}}^{n+1+ \tilde{\delta}}  N(\varphi_1, \varphi_2). 
\end{align}  
We get a contradiction because  $C_{\tilde{\delta}}$ can be chosen such that  $A_{\tilde{\delta}} C_{\tilde{\delta}}>1$.  Therefore  $E_{2\epsilon(\tilde{C}+2)}$ is empty. In other words, we have 
$$\varphi_1 - \varphi_2 \gtrsim -   \big(N(\varphi_1, \varphi_2)\big)^{1/(n+1+ \tilde{\delta})}.$$
By swapping the roles of $\varphi_1, \varphi_2$ we also get 
$$\varphi_2 - \varphi_1 \gtrsim  - \big(N(\varphi_1, \varphi_2)\big)^{1/(n+1+ \tilde{\delta})}.$$
This implies that 
\begin{align} \label{ine_boundLvocung}
\|\varphi_1 - \varphi_2\|_{L^{\infty}(X)} \lesssim  \big(N(\varphi_1, \varphi_2)\big)^{1/(n+1+ \tilde{\delta})} \lesssim  \| \varphi_1 - \varphi_2\|_{L^1(X)}^{\alpha/(n+1+ \tilde{\delta})}
\end{align}
which is 
$$\lesssim \dist(\mu_1, \mu_2)^{\alpha \beta 2^{-n}/(n+1+ \tilde{\delta})}.$$
Now applying Lemma \ref{le_danhgiachuanholder} to $f= \varphi_1-\varphi_2$ and using (\ref{ine_boundLvocung}) we obtain that  for any $\delta \in [0, \frac{2 \alpha}{n+1}-\beta),$ 
$$\|\varphi_1- \varphi_1\|_{\Cc^{\beta}} \lesssim  \|\varphi_1 - \varphi_2\|_{L^{\infty}(X)}^{\frac{\delta}{\beta+\delta}} \lesssim  \dist(\mu_1, \mu_2)^{\delta \alpha 2^{-n}(n+1+ \tilde{\delta})^{-1} \frac{\beta}{\beta+\delta}}.$$
Letting $\delta \rightarrow (\frac{2 \alpha}{n+1}-\beta)$ and $\tilde{\delta} \rightarrow 0$  gives the desired result.  The proof is finished.
\end{proof}

\begin{proposition} \label{pro_uniformrealsub} Let $M,\{Y_{\tau}\}_{\tau \in M}$ and $\Psi$ be as in Introduction. Then the superpotential of $\mu_\tau$ is H\"older continuous with uniformly bounded H\"older exponent and H\"older constant as $\tau$ varies in $M.$
\end{proposition}

\proof As already mentioned, the desired result can be deduced directly from  \cite{Vu_MA}. We briefly explain it here for the reader's convenience. We need to prove (\ref{le_ine_trunhau}) for $\mu_\tau$ instead of $\mu$ and the constants $C, \alpha$ there must be independent of $\tau.$ 

Since the problem is local, it is enough to work locally. Each $Y_\tau$ inherits the metric from $Y.$ Fix $\tau \in M$ and a point $a \in Y_\tau.$ The crucial point is that the data in \cite[Le. 3.1]{Vu_MA} can be chosen uniformly in $\tau, a.$ To be precise,  there exists a local chart $(W, \Phi)$ around $a$ in $X$ with $\Phi: W \rightarrow \C^n$  such that the following three properties holds:

$(i)$ $W \cap K_\tau$ contains a ball $\B_{Y_\tau}(a, r)$ of radius $r_a$ centered at $a$ of $Y_\tau,$ where  $r>0$ is a constant independent of $a, \tau,$

$(ii)$ $\|\Phi\|_{\cali{C}^3}$ and $\|\Phi^{-1}\|_{\cali{C}^3}$ are bounded by a constant independent of $a,\tau,$

$(iii)$ $\Phi(W\cap K)$ is the graph over the unit ball  $\B$ of $\R^n$ of a smooth map $h: \overline{\B} \rightarrow \R^n$ ($\C^n \approx \R^n +i \R^n$) such that $\|h\|_{\cali{C}^3}$ is bounded by a constant independent of $a, \tau$ and $D^j h(0)=0$ for $j=0,1,2.$
 
 By Property $(i),$ the number of local charts  $(W, \Phi)$   needed to cover $K_\tau$ can be chosen to be a fixed number for every $\tau.$ On a such local chart, every constant in  \cite[Pro. 3.7]{Vu_MA} can be chosen to be the same for every $\tau, a.$ Now the rest of the proof is done as in \cite{Vu_MA}. This gives us  constants $C,\alpha$ in (\ref{le_ine_trunhau}) independent of $\tau.$  The proof is finished.
\endproof

\bibliography{biblio_family_MA}
\bibliographystyle{siam}

\end{document}